\documentclass{article}

\usepackage[a4paper,top=2.54cm,bottom=2.54cm,left=3.17cm,right=3.17cm,%
            includehead,includefoot]{geometry}

\usepackage{amsmath,amssymb,amsfonts,amsthm}
\usepackage{graphicx}
\usepackage{subfigure}
\usepackage{float}
\usepackage{xcolor}
\usepackage[numbers,square,sort&compress]{natbib}
\usepackage{hyperref}
  \hypersetup{colorlinks,citecolor=blue,linkcolor=blue,breaklinks=true}
\usepackage{booktabs}
\usepackage{colortbl}
\usepackage{caption}
\usepackage{enumitem}
\usepackage{epstopdf}
\usepackage{algorithm}
\usepackage{algpseudocode}
\usepackage{array}
\usepackage{bbding}
\usepackage{fancyhdr} % 页眉
\usepackage{fancyvrb} % 摘录 Verbatim 环境
\usepackage{longtable}
\usepackage{listings}
\usepackage{rotating,rotfloat} % 提供 sidewayfigure 和 sidewaystable 环境横排图表
\usepackage{yhmath} % \wideparen 弧 \adots

\usepackage{fancyhdr}
\allowdisplaybreaks

\newtheorem{theorem}{Theorem}[section]
\newtheorem{lemma}{Lemma}[section]
\newtheorem{corollary}[theorem]{Corollary}
\newtheorem{proposition}[theorem]{Proposition}

\newtheorem{remark}{Remark}
\begin{document}
\title{On a non-local area-preserving curve flow}

\author{
  Zezhen Sun
    \thanks{Corresponding author. The Department of General Education, Anhui Water Conservancy Technical College, Hefei 231603, China. E-mail address: \texttt{52205500017@stu.ecnu.edu.cn} }
  \and
  Yuting Wu
    \thanks{School of Mathematical Sciences, East China Normal University, Shanghai 200241, China. E-mail address: \texttt{52215500001@stu.ecnu.edu.cn}}
}

\date{\today}
\maketitle
\begin{abstract}
  \noindent
  In this paper, we study a new area-preserving curvature flow for closed convex planar curves. This flow will decrease the length of the evolving curve and make the curve more and more circular during the evolution process. And finally, the curve converges to a finite circle in $C^{\infty}$ sense as time goes to infinity.
\end{abstract}

\section{Introduction}
The curvature flow of plane curves has received a lot of attention during the last few decades. The most famous one is the curve shortening flow in the plane studied by Gage \cite{Gage83,Gage84}, Gage-Hamilton \cite{Gage-Hamilton}, Grayson \cite{Grayson87,Grayson} and the others. Later, people started to study some non-local flows, such as area-preserving flows \cite{Gage85,Ma-Cheng,Mao-Pan-Wang,Sun,gpt2,Tsai-Wang} and the length-preserving flows \cite{Ma-zhu,Pan-Yang,gpt1,Tsai-Wang}, etc. In \cite{Lin-Tsai2012}, Lin-Tsai summarized the previous non-local flows as
\begin{equation}\label{fl}
  \left\{\begin{aligned}
    \frac{\partial X(u, t)}{\partial t}&=\big[f(\kappa(u,t))-\lambda(t)\big]N_{in}(u,t),\\
    X(u, 0)&=X_{0}(u),\quad u\in S^{1},
  \end{aligned}
  \right.
  \end{equation}
where $X_{0}(u)\subset\mathbb{R}^{2}$ is a given smooth closed curve, parameterized by $u\in S^{1}$, and $X(u, t):S^{1}\times[0,T)\to\mathbb{R}^{2}$ is a family of curves moving along its inward normal direction $
N_{in}(u,t)$ with given speed function $f(\kappa(u,t))-\lambda(t)$. $f(\kappa)$ is a function
of the curvature $\kappa(u,t)$ and $\lambda(t)$ is a function of time which depends on certain global (non-local) quantities, say the length $L(t)$, enclosed area $A(t)$, or other possible global quantities like the integral of curvatue over the entire curve.

 The aim of this paper is to study a new non-local flow for convex closed plane curves which is not included in the aforementioned types:
\begin{equation}\label{fl}
  \left\{\begin{aligned}
    \frac{\partial X(u, t)}{\partial t}&=\big(\kappa-\frac{\lambda(t)}{\kappa}\big)N(u,t),\\
    X(u, 0)&=X_{0}(u),\quad u\in S^{1},
  \end{aligned}
  \right.
  \end{equation}
 where $\kappa$ is the curvature at the point $X(u, t)$ and the nonlocal term $\lambda(t)=\frac{2\pi}{\oint\frac{1}{\kappa}ds}$ is used to keep enclosed area $A$ a constant. Our main result is given as follows.

\begin{theorem}\label{them}
A closed convex plane curve which evolves according to \eqref{fl} remains convex, decreases its length and preserves the enclosed area during the evolution process, and finally converges to a finite circle with
radius $\sqrt{\frac{A(0)}{\pi}}$ in $C^{\infty}$ sense as $t\to\infty$.
\end{theorem}

We organize the sections as follows. In Section 2, it is proved that the strictly convexity of the flow is preserved. It is also shown that the evolving curve converges to a circle
 in the Hausdorff metric under the assumption that the flow
 exists globally.  In Section 3,  it is proved that the evolution problem
 is equivalent to an initial value
 problem of a certain nonlinear differential equation system. Then by
 the Leray–Schauder fixed point principle, we obtain that this initial
 value problem is locally solvable.  In Section 4, we follow the standard ideas of curvature
 flow theory to prove the long-time existence of the evolving curve.  In Section 5, we show that the evolving curve converges to a circle in the $C^{\infty}$ metric.

\section{Final Shape of the Evolving Curve}
Let $X(u,t)=\big(x(u,t),y(u,t)\big):[a,b]\times[0,T)\to\mathbb{R}^{2}$ be a family of closed plane curves with initial curve $X(u, 0)=X_{0}(u):[a,b]\to\mathbb{R}^{2}$ being a convex closed curve which evolves under \eqref{fl}.

Let $g(u,t)=\left|X_{u}\right|=(x^{2}_{u}+y^{2}_{u})^{\frac{1}{2}}$ denote the metric along the curve, then the differential of arc-length is $ds = g(u,t)du$. The tangent vector $T$, normal vector $N$, direction angle $\theta$, curvature $\kappa$, length $L$ and area $A$ can be defined as the following equations:
  \begin{align*}
    T&=\frac{\partial X}{\partial s} = \frac{1}{g}\frac{\partial X}{\partial u}, \quad N = \frac{1}{\kappa}\frac{\partial T}{\partial s} = \frac{1}{\kappa g}\frac{\partial T}{\partial u},\qquad\,\theta = \angle (T,x),\quad
    \kappa =\frac{\partial \theta}{\partial s}= \frac{1}{g} \frac{\partial \theta}{\partial u},\\
      A(t)&=\frac{1}{2}\oint\,xdy-ydx=-\frac{1}{2}\oint \langle X,N\rangle ds, \quad
   L(t)=\int^{b}_{a} g(u,t)du = \oint ds.
  \end{align*}
  Since changing the tangential components of the velocity vector $X_{t}$ affects only the parametrization, not the geometric shape of the evolving curve (see \cite{Chou-Zhu,Gage-Hamilton,Grayson87}), one may choose a proper tangential component $\eta=-\frac{\partial}{\partial\theta}\big(\kappa-\frac{1}{\kappa}\frac{2\pi}{\oint\frac{1}{\kappa}ds}\big)$ which makes $\theta$ independent of time $t$, such that the geometric analysis of the evolving curve can be simplified, i.e., we may consider the following evolution problem which is equivalent to flow \eqref{fl}:
  \begin{equation}\label{fl2}
  \left\{\begin{aligned}
    \frac{\partial X(u, t)}{\partial t}&=\eta\,T+\big(\kappa-\frac{1}{\kappa}\frac{2\pi}{\oint\frac{1}{\kappa}ds}\big)N,\\
    X(u, 0)&=X_{0}(u).
  \end{aligned}
  \right.
  \end{equation}

\begin{lemma}\label{llaa}
(The monotonicity of area and length)Under the flow \eqref{fl2}, we have
\begin{equation*}
\frac{dA}{dt}=0,
\quad\frac{dL}{dt}\le0.
\end{equation*}
\end{lemma}
\begin{proof}
Gage’s variation formulae of length $L(t)$ and area $A(t)$ in \cite{Gage85} give us
\begin{align}
\frac{dA}{dt}&=-\oint\big(\kappa-\frac{1}{\kappa}\frac{2\pi}{\oint\frac{1}{\kappa}ds}\big)ds=
-\oint\kappa ds+\oint\frac{1}{\kappa}ds\frac{2\pi}{\oint\frac{1}{\kappa}ds}=0,\\
\label{lt}\frac{dL}{dt}&=-\oint\big(\kappa-\frac{1}{\kappa}\frac{2\pi}{\oint\frac{1}{\kappa}ds}\big)\kappa ds=
-\oint\kappa^{2} ds+\frac{2\pi L}{\oint\frac{1}{\kappa}ds}.
\end{align}
To see $\frac{dL}{dt}\le0$, we need the following two inequalities (see \cite{Gage83,Pan-Yang})
\begin{align}
\label{gi}\oint\kappa^{2}ds&\ge\frac{\pi L}{A}\quad\text{(Gage's inequality)},\\
\label{pyi}\oint\frac{1}{\kappa}ds&\ge\frac{L^{2}-2\pi A}{\pi}\quad\text{(Pan-Yang inequality)}.
\end{align}
So one obtains that
$$\frac{dL}{dt}\le-\frac{\pi L}{A}+\frac{2\pi^{2}L}{L^{2}-2\pi A}=\frac{\pi L(4\pi A-L^{2})}{(L^{2}-2\pi A)A}\le0,$$
where we have used the classical isoperimetric inequality $L^{2}-4\pi A\ge0$. The proof is done.
\end{proof}

\begin{remark}\label{lajie}
By the above Lemma and  $L^{2}-4\pi A\ge0$, we have
\begin{equation}\label{la}
 A(t)\equiv A(0)\quad and \quad \sqrt{4\pi A(0)}\le L(t)\le L(0),
\end{equation}
and the isoperimetric ratio $\frac{L^{2}(t)}{4\pi A(t)}$ is decreasing during the evolution process unless the initial curve $X_{0}$ is a circle (which is an equilibrium solution of the flow \eqref{fl2}). Moreover, from \eqref{pyi} it follows that
\begin{equation}
0<\lambda(t)\le\frac{2\pi^{2}}{L^{2}-2\pi A}\le\frac{2\pi^{2}}{4\pi A-2\pi A}=\frac{\pi}{A}=\frac{\pi}{A(0)}.
\end{equation}
\end{remark}

In the following, we will employ the maximum principle to show that if the initial curvature is positive, then the curvature of the evolving curve remains positive for all time.

\begin{lemma}(Preserved convexity)
A closed strictly convex plane curve evolving according to \eqref{fl2} remains so during the evolution process.
\end{lemma}
\begin{proof}
Just as Gage, Hamilton and others have done, we can derive the evolution equation of the curvature $\kappa$ as follows:
\begin{equation}\label{kt}
\frac{\partial\kappa}{\partial t}=\kappa^{2}\bigg[\big(\kappa-\frac{\lambda(t)}{\kappa}\big)_{\theta\theta}+\kappa-\frac{\lambda(t)}{\kappa}\bigg]
=\big(\kappa^{2}+\lambda(t)\big)\kappa_{\theta\theta}-\frac{2\lambda(t)}{\kappa}\kappa^{2}_{\theta}+\kappa^{3}-\lambda(t)\kappa.
\end{equation}
Now, let $U(\theta,t)=\kappa(\theta,t)e^{\mu t}$, where $\mu$ is a constant to be chosen later on. Then $U$ satisfies the following equation:
\begin{equation}\label{ut}
U_{t}=\big(\kappa^{2}+\lambda(t)\big)U_{\theta\theta}-\frac{2\lambda(t)}{U}U^{2}_{\theta}+(\kappa^{2}-\lambda(t)+\mu)U.
\end{equation}
 Let
$$U_{\min}(t)=\inf\{U(\theta,t)|0\le\theta\le2\pi\},$$
and suppose that there exists a $\xi$, satisfying $0<\xi<U_{\min}(0)=\kappa_{\min}(0),$ such that $U_{\min}(t)=\xi$ for some time $t>0$. Let $t_{0}=\inf\{t|U_{\min}(t)=\xi\}>0$, then,
the continuity of $U$ assures that this minimum $\xi$ is achieved for the first time at $(\theta_{0},t_{0})$ for some $\theta_{0}\in S^{1}$. Choose $\mu=\frac{\pi}{A(0)}+1$, then $\kappa^{2}-\lambda(t)+\mu>0$ at $(\theta_{0},t_{0})$, and at this point,
$$U_{t}\le0,\quad U_{\theta\theta}\ge0,\quad U_{\theta}=0\quad \text{and}\quad U=\xi>0,$$
which contradicts with the fact that $U$ satisfies \eqref{ut}. Hence we obtain that $U_{\min}(t)\ge U_{\min}(0)$ for all $t>0$. As a result,
\begin{equation}\label{kxj}
\kappa(\theta,t)\ge\kappa_{\min}(t)=U_{\min}(t)e^{-\mu t}\ge\kappa_{\min}(0)e^{-\mu t}>0.
\end{equation}
This completes the proof.
\end{proof}

We now consider the growth of the nonlocal term $\lambda(t)$.
\begin{lemma}
Under the flow \eqref{fl2}, we have
\begin{equation}
\frac{d\lambda}{dt}\ge0.
\end{equation}
\end{lemma}
\begin{proof}
Do the computation.
\begin{align*}
\frac{d\lambda}{dt}&=\bigg(\frac{2\pi}{\int^{2\pi}_{0}\frac{1}{\kappa^{2}}d\theta}\bigg)_{t}
=-\frac{2\pi\big(\int^{2\pi}_{0}\frac{1}{\kappa^{2}}d\theta\big)_{t}}{\big(\int^{2\pi}_{0}\frac{1}{\kappa^{2}}d\theta\big)^{2}}
=\frac{4\pi}{\big(\int^{2\pi}_{0}\frac{1}{\kappa^{2}}d\theta\big)^{2}}\int^{2\pi}_{0}\frac{\kappa_{t}}{\kappa^{3}}d\theta\\
&=\frac{\lambda^{2}}{\pi}\int^{2\pi}_{0}\frac{1}{\kappa}\bigg[\big(\kappa-\frac{\lambda}{\kappa}\big)_{\theta\theta}
+\kappa-\frac{\lambda}{\kappa}\bigg]d\theta\\
&=-\frac{\lambda^{2}}{\pi}\int^{2\pi}_{0}\big(\frac{1}{\kappa}\big)_{\theta}\big(\kappa-\frac{\lambda}{\kappa}\big)_{\theta}d\theta
-\frac{\lambda^{2}A_{t}}{\pi}\\
&=\frac{\lambda^{2}}{\pi}\int^{2\pi}_{0}\frac{\kappa^{2}_{\theta}}{\kappa^{2}}\big(1+\frac{\lambda}{\kappa^{2}}\big)\ge0.
\end{align*}
The proof is done.
\end{proof}

\begin{theorem}\label{hasc}
Under the evolution process, the isoperimetric deficit $L^{2}-4\pi A$ decreases and decays to zero exponentially. Furthermore, if the evolving curve does not develop singularities then it converges to a circle in the Hausdorff sense.
\end{theorem}
\begin{proof}
Compute that
\begin{equation*}
\frac{d}{dt}(L^{2}-4\pi A)=2LL_{t}=\frac{4\pi L^{2}}{\oint\frac{1}{\kappa}ds}-2L\oint\kappa^{2} ds.
\end{equation*}
The Cauchy–Schwartz inequality gives us
$$\oint\frac{1}{\kappa}ds=\int^{2\pi}_{0}\frac{1}{\kappa^{2}}d\theta
\ge\frac{1}{2\pi}\bigg(\int^{2\pi}_{0}\frac{1}{\kappa}d\theta\bigg)^{2}=\frac{L^{2}}{2\pi}.$$
Combining this with \eqref{gi} we obtain
$$\frac{d}{dt}(L^{2}-4\pi A)\le-\frac{\pi}{AL}(L^{2}-4\pi A)\le-\frac{\pi}{AL(0)}(L^{2}-4\pi A).$$
Hence,
$$L^{2}-4\pi A\le (L^{2}(0)-4\pi A(0))e^{-\frac{\pi}{AL(0)}t},$$
and as $t\to\infty,$ we have $L^{2}-4\pi A\to0.$
By the Bonnesen inequality \cite{Osserman-Bonnesn}
$$L^{2}-4\pi A\ge\pi^{2}(r_{out}(t)-r_{in}(t))^{2},$$
where $r_{in}(t)$ is the inradius of $X$ and $r_{out}(t)$ is the outradius of $X$, the evolving curve must converge to a round circle with
$$\lim\limits_{t\to\infty}r_{in}(t)=\lim\limits_{t\to\infty}r_{out}(t)=\sqrt{\frac{A(0)}{\pi}}.$$
\end{proof}
\begin{remark}
In fact, we also have
$$\lim\limits_{t\to\infty}\kappa(\theta,t)=\frac{1}{R}=\sqrt{\frac{\pi}{A(0)}}, \quad \lim\limits_{t\to\infty}\lambda(t)=\frac{1}{R^{2}}=\frac{\pi}{A(0)},$$
where $R$ is radius of the limiting circle.
\end{remark}
\section{Short-time existence of the flow}
In this section, we will reduce the evolution problem for the curve to an initial value
problem of a certain nonlinear differential equation system s and then
consider the short-time existence and uniqueness for the later initial value
problem.

Similar to the proof of (\cite{Jiang-Pan} Lemma 3.2), the following "closing condition" can easily be verified.
\begin{lemma}
If $\kappa_{0}(\theta)=\kappa(\theta,0)>0$ satisfies $\int^{2\pi}_{0}\frac{e^{i\theta}}{\kappa_{0}(\theta)}d\theta=0$, then for each $t>0$, the solution $\kappa(\theta,t)$ to \eqref{kt} satisfies $\int^{2\pi}_{0}\frac{e^{i\theta}}{\kappa(\theta,t)d\theta}=0.$
\end{lemma}
Next, we will show that the curve evolution problem can be reduced to an initial value problem
of some nonlinear differential-integral equations.
\begin{theorem}
Geometric evolution problem \eqref{fl2} is equivalent to the following initial
value problem: Find $\kappa=\kappa(\theta,t): S^{1}\times[0,\infty)\to\mathbb{R^{+}}, L=L(t):[0,\infty)\to\mathbb{R^{+}}$ such that

$(1)\,\,\kappa\in C^{2+\sigma,1+\frac{\sigma}{2}}((S^{1}\times[0,T^{\ast})), L\in C^{1}([0,T^{\ast}))$ for all $T^{\ast}>0$.

$(2)$\begin{align*}
\frac{\partial\kappa}{\partial t}
&=\big(\kappa^{2}+\frac{2\pi}{\int^{2\pi}_{0}\frac{1}{\kappa^{2}}d\theta}\big)\kappa_{\theta\theta}-\frac{4\pi}{\kappa\int^{2\pi}_{0}\frac{1}{\kappa^{2}}d\theta}\kappa^{2}_{\theta}
+\kappa^{3}-\frac{2\pi}{\int^{2\pi}_{0}\frac{1}{\kappa^{2}}d\theta}\kappa,\\
\frac{dL}{dt}&=-\int^{2\pi}_{0}\kappa d\theta+\frac{2\pi L}{\int^{2\pi}_{0}\frac{1}{\kappa^{2}}d\theta}.
\end{align*}

$(3)(a)\kappa_{0}(\theta)=\kappa(\theta,0)\in C^{1+\sigma}(S^{1}), \kappa_{0}(\theta)>0$ and $\int^{2\pi}_{0}\frac{e^{i\theta}}{\kappa_{0}(\theta)}d\theta=0$; $(b)L(0)=\int^{2\pi}_{0}\frac{d\theta}{\kappa_{0}(\theta)}>0.$
\end{theorem}
\begin{proof}
Equations \eqref{lt} and \eqref{kt} tell us that given a solution to the evolution problem \eqref{fl2}, the length and curvature functions of the evolving curves, expressed in $\theta$ and $t$ coordinates, will satisfy $(2)$ and $(3)$.

Conversely, given a solution $\kappa(\theta,t)$ and $L(t)$ to the initial value problem of the differential-integral system, then one can define a curve
$$\bar{X}(\theta,t)=(x(\theta,t),y(\theta,t))+(C_{1}(t),C_{2}(t))$$
with
\begin{equation}\label{xy}
x(\theta,t)=\int^{\theta}_{0}\frac{\cos\phi}{\kappa(\phi,t)}d\phi,\quad y(\theta,t)=\int^{\theta}_{0}\frac{\sin\phi}{\kappa(\phi,t)}d\phi,
\end{equation}
$$C_{1}(t)=-\int^{t}_{0}\frac{\partial}{\partial\theta}\bigg(\kappa-\frac{2\pi}{\kappa\int^{2\pi}_{0}\frac{1}{\kappa^{2}}d\theta}\bigg)(0,\tau)d\tau\quad
\text{and}\quad C_{2}(t)=\int^{t}_{0}\bigg(\kappa-\frac{2\pi}{\kappa\int^{2\pi}_{0}\frac{1}{\kappa^{2}}d\theta}\bigg)(0,\tau)d\tau.$$
And this curve satisfies the evolution equation
$$\frac{\partial \bar{X}}{\partial t}=-\frac{\partial}{\partial\theta}\bigg(\kappa-\frac{2\pi}{\kappa\int^{2\pi}_{0}\frac{1}{\kappa^{2}}d\theta}\bigg)T
+\big(\kappa-\frac{2\pi}{\kappa\int^{2\pi}_{0}\frac{1}{\kappa^{2}}d\theta}\big)N.$$
A direct calculation gives that the curvature of the curve $\bar{X}(\theta,t)$ is exactly $\kappa(\theta,t)$.

Next, we check that the length $\bar{L}(t)$ of the curve $\bar{X}(\theta,t)$ is $L(t)$. By \eqref{xy}, we have
\begin{align*}
\bar{L}(t)&=\int^{2\pi}_{0}\sqrt{x^{2}_{\theta}+y^{2}_{\theta}}d\theta=\int^{2\pi}_{0}\sqrt{\frac{\cos^{2}\theta}{\kappa^{2}}
+\frac{\sin^{2}\theta}{\kappa^{2}}}d\theta=\int^{2\pi}_{0}\frac{1}{\kappa}d\theta,\\
\bar{L}(0)&=\int^{2\pi}_{0}\frac{1}{\kappa_{0}}d\theta=L(0).
\end{align*}
While
\begin{align*}
\bar{L}_{t}&=\int^{2\pi}_{0}\big(\frac{1}{\kappa}\big)_{t}d\theta=\int^{2\pi}_{0}\bigg[\frac{1}{\kappa_{0}(\theta)}-
\int^{t}_{0}\bigg(\big(\kappa-\frac{2\pi}{\kappa\int^{2\pi}_{0}\frac{1}{\kappa^{2}}d\theta}\big)_{\theta\theta}
+\kappa-\frac{2\pi}{\kappa\int^{2\pi}_{0}\frac{1}{\kappa^{2}}d\theta}\bigg)d\tau\bigg]_{t}d\theta\\
&=-\int^{2\pi}_{0}\bigg(\big(\kappa-\frac{2\pi}{\kappa\int^{2\pi}_{0}\frac{1}{\kappa^{2}}d\theta}\big)_{\theta\theta}
+\kappa-\frac{2\pi}{\kappa\int^{2\pi}_{0}\frac{1}{\kappa^{2}}d\theta}\bigg)d\theta\\
&=\frac{2\pi L}{\int^{2\pi}_{0}\frac{1}{\kappa^{2}}d\theta}-\int^{2\pi}_{0}\kappa d\theta=L_{t},
\end{align*}
which is desired.
\end{proof}

\begin{lemma}
For $t<t_{0}=\frac{1}{m_{0}}\big[\int^{2\pi}_{0}\kappa^{2}_{0}(\theta)d\theta\big]^{-1}$, where $m_{0}$ is an absolute constant, we have
$$\int^{2\pi}_{0}\kappa^{2}(\theta,t)d\theta\le\big(M(t)+1\big)\int^{2\pi}_{0}\kappa^{2}_{0}(\theta)d\theta,$$
where $M(t)=\frac{1}{1-cm_{0}t}, c=\int^{2\pi}_{0}\kappa^{2}_{0}(\theta)d\theta.$ In particular, for $t\in[0,\frac{t_{0}}{2}]$, we obtain
\begin{equation}\label{k2sj}
\int^{2\pi}_{0}\kappa^{2}(\theta,t)d\theta\le3\int^{2\pi}_{0}\kappa^{2}_{0}(\theta)d\theta.
\end{equation}
\end{lemma}
\begin{proof}
By the evolution equation of curvature $\kappa$, one gets
\begin{equation}\label{kkt}
\kappa\kappa_{t}=\big(\kappa^{3}+\frac{2\pi}{\int^{2\pi}_{0}\frac{1}{\kappa^{2}}d\theta}\kappa\big)\kappa_{\theta\theta}
-\frac{4\pi}{\int^{2\pi}_{0}\frac{1}{\kappa^{2}}d\theta}\kappa^{2}_{\theta}
+\kappa^{4}-\frac{2\pi}{\int^{2\pi}_{0}\frac{1}{\kappa^{2}}d\theta}\kappa^{2}.
\end{equation}
Integrating \eqref{kkt} over $S^{1}\times[0,t]$ yields
\begin{align*}
\frac{1}{2}\int^{2\pi}_{0}&\kappa^{2}d\theta-\frac{1}{2}\int^{2\pi}_{0}\kappa^{2}_{0}d\theta\\
=&\frac{1}{2}\int^{t}_{0}\big(
\int^{2\pi}_{0}\kappa^{2}d\theta\big)_{t}dt=\int^{t}_{0}\int^{2\pi}_{0}\kappa\kappa_{t}d\theta dt\\
=&\int^{t}_{0}\int^{2\pi}_{0}\bigg[\big(\kappa^{3}+\frac{2\pi}{\int^{2\pi}_{0}\frac{1}{\kappa^{2}}d\theta}\kappa\big)
\kappa_{\theta\theta}-\frac{4\pi}{\int^{2\pi}_{0}\frac{1}{\kappa^{2}}d\theta}\kappa^{2}_{\theta}
+\kappa^{4}-\frac{2\pi}{\int^{2\pi}_{0}\frac{1}{\kappa^{2}}d\theta}\kappa^{2}\bigg]d\theta dt\\
=&\int^{t}_{0}\int^{2\pi}_{0}\kappa^{3}\kappa_{\theta\theta}d\theta dt-\int^{t}_{0}\int^{2\pi}_{0}\frac{6\pi}{\int^{2\pi}_{0}\frac{1}{\kappa^{2}}d\theta}\kappa^{2}_{\theta}d\theta dt
+\int^{t}_{0}\int^{2\pi}_{0}\kappa^{4}d\theta dt\\
&-\int^{t}_{0}\int^{2\pi}_{0}\frac{2\pi}{\int^{2\pi}_{0}\frac{1}{\kappa^{2}}d\theta}\kappa^{2}d\theta dt\\
\le&\int^{t}_{0}\int^{2\pi}_{0}\kappa^{3}\kappa_{\theta\theta}d\theta dt+\int^{t}_{0}\int^{2\pi}_{0}\kappa^{4}d\theta dt\\
=&-\frac{3}{4}\int^{t}_{0}\int^{2\pi}_{0}\big((\kappa^{2})_{\theta}\big)^{2}d\theta dt+\int^{t}_{0}\int^{2\pi}_{0}\kappa^{4}d\theta dt.
\end{align*}
Set $\kappa^{2}=\omega$, then
\begin{equation}\label{ome}
\frac{1}{2}\int^{2\pi}_{0}\omega d\theta+\frac{3}{4}\int^{t}_{0}\int^{2\pi}_{0}\omega^{2}_{\theta}d\theta dt
\le\frac{1}{2}\int^{2\pi}_{0}\omega_{0} d\theta+\int^{t}_{0}\int^{2\pi}_{0}\omega^{2}d\theta dt.
\end{equation}
We shall use the Nirenberg's interpolation inequality (see \cite{oaso},pp.62-63): For any function $f$ satisfying $\int^{2\pi}_{0}fd\theta=0$, there holds
$$\Vert f\Vert_{2}\le m\Vert f_{\theta}\Vert^{\frac{1}{3}}_{2}\Vert f\Vert^{\frac{2}{3}}_{1},$$
where $m=(1+\frac{1}{2})^{\frac{1}{3}}=(\frac{3}{2})^{\frac{1}{3}}$. Hence for any function $v$, $v-\frac{1}{2\pi}\int^{2\pi}_{0}vd\theta=f$ satisfies the condition that $\int^{2\pi}_{0}fd\theta=0$. Using the Nirenberg’s inequality yields
$$\bigg\Vert v-\frac{1}{2\pi}\int^{2\pi}_{0}vd\theta\bigg\Vert_{2}\le m\Vert v_{\theta}\Vert^{\frac{1}{3}}_{2}\cdot
\bigg\Vert v-\frac{1}{2\pi}\int^{2\pi}_{0}vd\theta\bigg\Vert^{\frac{2}{3}}_{1}.$$
While
\begin{align*}
\bigg\Vert v-\frac{1}{2\pi}\int^{2\pi}_{0}vd\theta\bigg\Vert_{2}&\ge\Vert v\Vert_{2}-\frac{1}{\sqrt{2\pi}}\int^{2\pi}_{0}\lvert v\lvert d\theta,\\
\bigg\Vert v-\frac{1}{2\pi}\int^{2\pi}_{0}vd\theta\bigg\Vert_{1}&=\int^{2\pi}_{0}\bigg\lvert v-\frac{1}{2\pi}\int^{2\pi}_{0}vd\theta\bigg\lvert d\theta\le\int^{2\pi}_{0}\lvert v\lvert d\theta+\bigg\lvert \int^{2\pi}_{0} v d\theta\bigg\lvert\le2\Vert v\Vert_{1}.
\end{align*}
Therefore, for any function $v$, we have
\begin{align*}
\Vert v\Vert_{2}&\le\bigg\Vert v-\frac{1}{2\pi}\int^{2\pi}_{0}vd\theta\bigg\Vert_{2}+\frac{1}{\sqrt{2\pi}}\Vert v\Vert_{1}\\
&\le m\Vert v_{\theta}\Vert^{\frac{1}{3}}_{2}\cdot\bigg\Vert v-\frac{1}{2\pi}\int^{2\pi}_{0}vd\theta\bigg\Vert^{\frac{2}{3}}_{1}+\frac{1}{\sqrt{2\pi}}\Vert v\Vert_{1}\\
&\le m\Vert v_{\theta}\Vert^{\frac{1}{3}}_{2}\cdot(2^{\frac{2}{3}}\Vert v\Vert^{\frac{2}{3}}_{1})+\frac{1}{\sqrt{2\pi}}\Vert v\Vert_{1}\\
&=m_{1}\Vert v_{\theta}\Vert^{\frac{1}{3}}_{2}\Vert v\Vert^{\frac{2}{3}}_{1}+m_{2}\Vert v\Vert_{1},
\end{align*}
where $m_{1}=2^{\frac{2}{3}}\cdot m, $ and $m_{2}=\frac{1}{\sqrt{2\pi}}$. Then we use Young’s inequality to obtain for any $\epsilon>0$
\begin{align*}
\Vert v\Vert_{2}&\le(\epsilon^{2}m_{1}\Vert v_{\theta}\Vert_{2})^{\frac{1}{3}}\big(\frac{1}{\epsilon}m_{1}\Vert v\Vert_{1}\big)^{\frac{2}{3}}+m_{2}\Vert v\Vert_{1}\le\frac{1}{3}\epsilon^{2}m_{1}\Vert v_{\theta}\Vert_{2}
+\frac{2}{3}\big(\frac{m_{1}}{\epsilon}\Vert v\Vert_{1}\big)+m_{2}\Vert v\Vert_{1}.
\end{align*}
Hence we get
\begin{equation}
\Vert v\Vert^{2}_{2}\le\frac{2}{9}\epsilon^{4}m^{2}_{1}\Vert v_{\theta}\Vert^{2}_{2}+2\big(\frac{2m_{1}}{3\epsilon}+m_{2}\big)^{2}\Vert v\Vert^{2}_{1}.
\end{equation}
Combining this with \eqref{ome} we obtain
\begin{align*}
\frac{1}{2}\int^{2\pi}_{0}&\omega d\theta+\big(\frac{3}{4}-\frac{2}{9}\epsilon^{4}m^{2}_{1}\big)\int^{t}_{0}\int^{2\pi}_{0}\omega^{2}_{\theta}d\theta dt\\
&\le2\big(\frac{2m_{1}}{3\epsilon}+m_{2}\big)^{2}\int^{t}_{0}\bigg(\int^{2\pi}_{0}\omega d\theta\bigg)^{2}dt+\frac{1}{2}\int^{2\pi}_{0}\omega_{0} d\theta.
\end{align*}
Choosing $\epsilon>0$ such that $\frac{3}{4}-\frac{2}{9}\epsilon^{4}m^{2}_{1}>0$ and letting
$$m_{0}=4\big(\frac{2m_{1}}{3\epsilon}+m_{2}\big)^{2},$$
then one has
\begin{equation}\label{ome2}
\int^{2\pi}_{0}\omega d\theta\le m_{0}\int^{t}_{0}\bigg(\int^{2\pi}_{0}\omega d\theta\bigg)^{2}dt
+\int^{2\pi}_{0}\omega_{0} d\theta.
\end{equation}
Set $F(t)=\int^{t}_{0}\big(\int^{2\pi}_{0}\omega d\theta\big)^{2}dt$ and $c=\int^{2\pi}_{0}\omega_{0} d\theta=
\int^{2\pi}_{0}\kappa^{2}_{0}(\theta) d\theta$. Then we have $\frac{dF}{dt}=\big(\int^{2\pi}_{0}\omega d\theta\big)^{2},F(0)=0$, which together with \eqref{ome2} leads to
$$\big(\frac{dF}{dt}\big)^{\frac{1}{2}}\le m_{0}F(t)+c,\quad F(0)=0.$$
This implies
$$F(t)\le\frac{c^{2}t}{1-m_{0}ct}.$$
Let $t_{0}=\frac{1}{cm_{0}}=\frac{1}{m_{0}}\big(\int^{2\pi}_{0}\kappa^{2}_{0}(\theta) d\theta\big)^{-1},$ then for $t<t_{0}$, we have
$$\int^{t}_{0}\bigg(\int^{2\pi}_{0}\omega d\theta\bigg)^{2}dt=F(t)\le c^{2}M(t)t<c^{2}M(t)t_{0},$$
where $M(t)=\frac{1}{1-cm_{0}t}>0$ for $0\le t<t_{0}$, $M(t)\to\infty$ as $t\to t_{0}$. Taking this into \eqref{ome2} yields that for $0\le t<t_{0}$,
$$\int^{2\pi}_{0}\omega d\theta\le m_{0}F(t)+c\le c^{2}m_{0}M(t)t+c\le(cm_{0}t_{0}M(t)+1)c=(M(t)+1)c.$$
Note that $M(t)=\frac{1}{1-cm_{0}t}>0$ is increasing for $0\le t<t_{0}$, for any $t\le\frac{t_{0}}{2}$, $M(t)\le M(\frac{t_{0}}{2})=2$ and therefore for $0\le t\le\frac{t_{0}}{2}$,
$$\int^{2\pi}_{0}\kappa^{2}(\theta,t)d\theta=\int^{2\pi}_{0}\omega(\theta,t) d\theta\le3\int^{2\pi}_{0}\omega_{0}(\theta) d\theta=3\int^{2\pi}_{0}\kappa^{2}_{0}(\theta)d\theta.$$
\end{proof}

\begin{lemma}\label{k2tj}
We can find a constant $C$ depending only on the initial data such that
$$\int^{2\pi}_{0}\kappa^{2}_{\theta}d\theta\le C$$
holds for $t\in[0,\frac{t_{0}}{2}]$.
\end{lemma}
\begin{proof}
By the evolution equation of curvature $\kappa$, we have
\begin{align*}
\frac{d}{dt}\int^{2\pi}_{0}&\bigg[\big(\kappa-\frac{\lambda}{\kappa}\big)^{2}-\big(\kappa-\frac{\lambda}{\kappa}\big)^{2}_{\theta}\bigg]d\theta\\
=&\int^{2\pi}_{0}\bigg[2\big(\kappa-\frac{\lambda}{\kappa}\big)\big(\kappa-\frac{\lambda}{\kappa}\big)_{t}
-2\big(\kappa-\frac{\lambda}{\kappa}\big)_{\theta}\big(\kappa-\frac{\lambda}{\kappa}\big)_{\theta t}\bigg]d\theta\\
=&\int^{2\pi}_{0}\bigg[2\big(\kappa-\frac{\lambda}{\kappa}\big)\big(\kappa-\frac{\lambda}{\kappa}\big)_{t}
+2\big(\kappa-\frac{\lambda}{\kappa}\big)_{\theta\theta}\big(\kappa-\frac{\lambda}{\kappa}\big)_{t}\bigg]d\theta\\
=&2\int^{2\pi}_{0}\bigg[\big(\kappa-\frac{\lambda}{\kappa}\big)+\big(\kappa-\frac{\lambda}{\kappa}\big)_{\theta\theta}\bigg]
\big(\kappa-\frac{\lambda}{\kappa}\big)_{t}d\theta\\
=&2\int^{2\pi}_{0}\frac{\kappa_{t}}{\kappa^{2}}\big[\kappa_{t}-\big(\frac{\lambda}{\kappa}\big)_{t}\big]
=2\int^{2\pi}_{0}\frac{\kappa_{t}}{\kappa^{2}}\big(\kappa_{t}+\frac{\kappa_{t}\lambda}{\kappa^{2}}-\frac{\lambda_{t}}{\kappa}\big)d\theta\\
\ge&-2\lambda_{t}\int^{2\pi}_{0}\frac{\kappa_{t}}{\kappa^{3}}d\theta=-\frac{2\pi\lambda^{2}_{t}}{\lambda^{2}}
\ge-\frac{2\pi\lambda^{2}_{t}}{\lambda(0)}.
\end{align*}
By integration, we have
\begin{align*}
\int^{2\pi}_{0}\bigg[\big(\kappa-\frac{\lambda}{\kappa}\big)^{2}&-\big(\kappa-\frac{\lambda}{\kappa}\big)^{2}_{\theta}\bigg]d\theta
-\int^{2\pi}_{0}\bigg[\big(\kappa(\theta,0)-\frac{\lambda(0)}{\kappa(\theta,0)}\big)^{2}
-\big(\kappa(\theta,0)-\frac{\lambda(0)}{\kappa(\theta,0)}\big)^{2}_{\theta}\bigg]d\theta\\
&\ge-\frac{2\pi}{\lambda(0)}\int^{t}_{0}\lambda^{2}_{t}d\tau\triangleq-g(t).
\end{align*}
Note that $g(t)$ is increasing for $0\le t\le\frac{t_{0}}{2}$, then we have $g(t)\le g(\frac{t_{0}}{2})\triangleq\Omega$. Hence
\begin{align*}
\int^{2\pi}_{0}\big(\kappa-\frac{\lambda}{\kappa}\big)^{2}_{\theta}d\theta&\le
\int^{2\pi}_{0}\big(\kappa-\frac{\lambda}{\kappa}\big)^{2}d\theta-\int^{2\pi}_{0}\bigg[\big(\kappa(\theta,0)-\frac{\lambda(0)}{\kappa(\theta,0)}\big)^{2}
-\big(\kappa(\theta,0)-\frac{\lambda(0)}{\kappa(\theta,0)}\big)^{2}_{\theta}\bigg]d\theta+\Omega
\end{align*}
Letting $D=\max\{0,\Omega-\int^{2\pi}_{0}\bigg[\big(\kappa(\theta,0)-\frac{\lambda(0)}{\kappa(\theta,0)}\big)^{2}
-\big(\kappa(\theta,0)-\frac{\lambda(0)}{\kappa(\theta,0)}\big)^{2}_{\theta}\bigg]d\theta\}$ yields
\begin{equation}\label{kkk}
\int^{2\pi}_{0}\big(\kappa-\frac{\lambda}{\kappa}\big)^{2}_{\theta}d\theta\le
\int^{2\pi}_{0}\big(\kappa-\frac{\lambda}{\kappa}\big)^{2}d\theta+D.
\end{equation}
Therefore,
\begin{align*}
\int^{2\pi}_{0}\kappa^{2}_{\theta}d\theta&\le\int^{2\pi}_{0}\big(\kappa_{\theta}+\frac{\lambda\kappa_{\theta}}{\kappa^{2}}\big)^{2}d\theta
=\int^{2\pi}_{0}\big(\kappa-\frac{\lambda}{\kappa}\big)^{2}_{\theta}d\theta\le\int^{2\pi}_{0}\big(\kappa-\frac{\lambda}{\kappa}\big)^{2}d\theta+D\\
&\le\int^{2\pi}_{0}\kappa^{2}d\theta+\int^{2\pi}_{0}\frac{\lambda^{2}}{\kappa^{2}}d\theta+D.
\end{align*}
By combining \eqref{kxj} and \eqref{k2sj}, we obtain
$$\int^{2\pi}_{0}\kappa^{2}_{\theta}d\theta\le3\int^{2\pi}_{0}\kappa^{2}_{0}d\theta+\frac{\pi^{2}}{A^{2}(0)}\frac{e^{\mu t_{0}}}{\kappa^{2}_{\min}(0)}+D\triangleq C.$$
\end{proof}

\begin{lemma}
Let $\kappa_{\max}(t)=\sup\{\kappa(\theta,t)|\theta\in[0,2\pi]\}$, for all $t\in[0,\frac{t_{0}}{2}]$, there exists a constant $\tilde{M}$ depending only on the initial data such that
$$\kappa_{\max}\le\tilde{M}.$$
\end{lemma}
\begin{proof}
By Lemma \ref{k2tj} and the Sobolev embedding theorem, one gets
$$\kappa^{2}_{\max}\le\tilde{c}\int^{2\pi}_{0}(\kappa^{2}+\kappa^{2}_{\theta})d\theta\le3\tilde{c}\int^{2\pi}_{0}\kappa^{2}_{0}
d\theta+\tilde{c}C\triangleq\tilde{M}^{2},$$
where $\tilde{c}$ is a constant.
\end{proof}
By mimicking Pan \cite{pp}, we can apply the classical Leray-Schauder fixed point theory to prove the short-time existence of the flow. Since this part is standard, we omit the details and just state:
\begin{theorem}
Suppose $\kappa(\theta,0)=\kappa_{0}(\theta)\in C^{2+\sigma}(S^{1})$ and there exist positive constants $D_{1}$ and $D_{2}$ such that $D_{1}\le\kappa_{0}(\theta)\le D_{2}$. Then there exists $T_{0}>0$ such that there are unique solutions $\kappa\in C^{2+\sigma,1+\frac{\sigma}{2}}(Q_{0}), L\in C^{1}([0,T_{0}])$ to the following initial value problem:
\begin{align*}
\frac{\partial\kappa}{\partial t}
&=\big(\kappa^{2}+\frac{2\pi}{\int^{2\pi}_{0}\frac{1}{\kappa^{2}}d\theta}\big)\kappa_{\theta\theta}-\frac{4\pi}{\kappa\int^{2\pi}_{0}\frac{1}{\kappa^{2}}d\theta}\kappa^{2}_{\theta}
+\kappa^{3}-\frac{2\pi}{\int^{2\pi}_{0}\frac{1}{\kappa^{2}}d\theta}\kappa,\\
\frac{dL}{dt}&=-\int^{2\pi}_{0}\kappa d\theta+\frac{2\pi L}{\int^{2\pi}_{0}\frac{1}{\kappa^{2}}d\theta},\\
\kappa(\theta,0)&=\kappa_{0}(\theta),\\
L(0)&=L_{0}=\int^{2\pi}_{0}\frac{d\theta}{\kappa_{0}(\theta)}>0,
\end{align*}
where $Q_{0}=S^{1}\times[0,T_{0}]$.
\end{theorem}

\section{Long-time existence of the flow}
In this section, we follow \cite{Gage-Hamilton,Jiang-Pan} to derive key estimates of the curvature of the evolving curve. Define the median curvature $\kappa^{\ast}$ as
\begin{equation}
\kappa^{\ast}(t)=\sup\{\alpha|\kappa(\theta,t)>\alpha\,\text{on some interval of length}\,\pi\}.
\end{equation}
For any convex closed curve in the plane, Proposition 4.3.2 in \cite{Gage-Hamilton} tells us that the inequality
$\kappa^{\ast}(t)<\frac{L}{A}$ holds. By \eqref{la}, we have
\begin{proposition}(Geometric estimate)
Under the flow \eqref{fl2}, $\kappa^{\ast}(t)$ is uniformly bounded,i.e.,
\begin{equation}\label{ge1}
\kappa^{\ast}(t)\le\frac{L(0)}{A(0)},\quad\forall t\in[0,T^{\ast}).
\end{equation}
\end{proposition}

\begin{proposition}(Integral estimate)
If $\kappa^{\ast}(t)$ is bounded on $[0,T^{\ast})$, then $\int^{2\pi}_{0}\ln\kappa(\theta,t)d\theta$ is bounded on $[0,T^{\ast})$.
\end{proposition}
\begin{proof}
By \eqref{kxj}, we have for all $t>0$
\begin{align*}
\int^{2\pi}_{0}\ln\kappa(\theta,t)d\theta&\ge\int^{2\pi}_{0}\ln\kappa_{\min}(t)d\theta=2\pi\ln\kappa_{\min}(t)
\ge2\pi\ln[\kappa_{\min}(0)e^{-\mu t}]\\
&=2\pi\ln\kappa_{\min}(0)-\big(\frac{2\pi^{2}}{A(0)}+2\pi\big)t.
\end{align*}
On the other hand, from the evolution equation of curvature and integration by parts we obtain
\begin{align}
\frac{d}{dt}\int^{2\pi}_{0}\ln\kappa(\theta,t)d\theta&=\int^{2\pi}_{0}\bigg[\kappa\big(\kappa-\frac{\lambda}{\kappa}\big)_{\theta\theta}
+\kappa\big(\kappa-\frac{\lambda}{\kappa}\big)\bigg]d\theta\notag\\
&=-\int^{2\pi}_{0}\kappa_{\theta}\big(\kappa-\frac{\lambda}{\kappa}\big)_{\theta}d\theta+
\int^{2\pi}_{0}\kappa\big(\kappa-\frac{\lambda}{\kappa}\big)d\theta\notag\\
&=-\int^{2\pi}_{0}\kappa^{2}_{\theta}d\theta-\int^{2\pi}_{0}\frac{\lambda\kappa^{2}_{\theta}}{\kappa^{2}}d\theta+
\int^{2\pi}_{0}\kappa^{2}d\theta-2\pi\lambda\notag\\
\label{ge2}&\le\int^{2\pi}_{0}(-\kappa^{2}_{\theta}+\kappa^{2})d\theta.
\end{align}
The open set $V=\{\theta|\kappa(\theta,t)>\kappa^{\ast}(t)\}$ can be written uniquely as the union of a countable number of disjoint open intervals $I_{i}$, where the length of each interval must be less than or equal to $\pi$.
On each interval $I_{i}$, we have
\begin{align}
\int_{\bar{I}_{i}}(-\kappa^{2}_{\theta}+\kappa^{2})d\theta&=\int_{\bar{I}_{i}}\big[-(\kappa-\kappa^{\ast})^{2}_{\theta}+\kappa^{2}\big]d\theta
\le\int_{\bar{I}_{i}}\big[-(\kappa-\kappa^{\ast})^{2}+\kappa^{2}\big]d\theta\notag\\
\label{ge3}&=\int_{\bar{I}_{i}}\big[2\kappa\kappa^{\ast}-(\kappa^{\ast})^{2}\big]d\theta\le2\kappa^{\ast}\int_{\bar{I}_{i}}\kappa d\theta
\end{align}
where we have used the Wirtinger’s inequality. On the complement of $V$, $V^{c}$, we have the estimate $\kappa\le\kappa^{\ast}$ and
\begin{equation}\label{ge4}
\int_{V^{c}}(-\kappa^{2}_{\theta}+\kappa^{2})d\theta\le\kappa^{\ast}\int_{V^{c}}\kappa d\theta.
\end{equation}
Combining with \eqref{ge1}, \eqref{ge2}, \eqref{ge3} and \eqref{ge4}, we have
$$\frac{d}{dt}\int^{2\pi}_{0}\ln\kappa(\theta,t)d\theta\le2\kappa^{\ast}\int^{2\pi}_{0}\kappa d\theta\le\frac{2L(0)}{A(0)}\int^{2\pi}_{0}\kappa d\theta.$$
From the evolution equation of $L$, we obtain
$$\frac{d}{dt}\int^{2\pi}_{0}\ln\kappa(\theta,t)d\theta\le\frac{2L(0)}{A(0)}(L\lambda-L_{t})\le
\frac{2\pi L^{2}(0)}{A^{2}(0)}-\frac{2L(0)}{A(0)}L_{t}.$$
Integrating the above expression yields
\begin{align*}
\int^{2\pi}_{0}\ln\kappa(\theta,t)d\theta&\le\int^{2\pi}_{0}\ln\kappa_{0}(\theta)d\theta+\frac{2\pi L^{2}(0)}{A^{2}(0)}t-\frac{2L(0)}{A(0)}\big(L(t)-L(0)\big)\\
&\le\int^{2\pi}_{0}\ln\kappa_{0}(\theta)d\theta+\frac{2\pi L^{2}(0)}{A^{2}(0)}T^{\ast}+\frac{2L^{2}(0)}{A(0)}.
\end{align*}
The proof is done.
\end{proof}

\begin{lemma}\label{lnk1}
If $\int^{2\pi}_{0}\ln\kappa(\theta,t)d\theta$ is bounded on $[0,T^{\ast})$, then for any $\delta>0$ we can find a positive constant $C$ such that $\kappa(\theta,t)\le C$ except on intervals of length less than or equal to $\delta$.
\end{lemma}
\begin{proof}
The proof proceeds by contradiction. If $\kappa>C$ on $a\le\theta\le b$ and $b-a>\delta$, then
\begin{align*}
\int^{2\pi}_{0}\ln\kappa(\theta,t)d\theta&=\int^{b}_{a}\ln\kappa(\theta,t)d\theta+\int_{[0,2\pi]\setminus[a,b]}\ln\kappa(\theta,t)d\theta\\
&>(b-a)\ln C+\big(2\pi-(b-a)\big)\ln\kappa_{\min}(t)\\
&=2\pi\ln\kappa_{\min}(t)+(b-a)\big(\ln C-\ln\kappa_{\min}(t)\big)\\
&>2\pi\ln\kappa_{\min}(t)+\delta\big(\ln C-\ln\kappa_{\min}(t)\big)\\
&=\delta\ln C+(2\pi-\delta)\ln\kappa_{\min}(t)\\
&\ge\delta\ln C+(2\pi-\delta)\ln\kappa_{\min}(0)-(2\pi-\delta)\big(\frac{\pi}{A(0)}+1\big)t,
\end{align*}
which leads to a contradiction when $C$ is large enough.
\end{proof}

\begin{lemma}(Pointwise estimate)
If $\int^{2\pi}_{0}\ln\kappa(\theta,t)d\theta$ is bounded on $[0,T^{\ast})$, then $\kappa(\theta,t)$ is uniformly bounded on $S^{1}\times[0,T^{\ast})$.
\end{lemma}
\begin{proof}
For any given $\delta$, by Lemma \ref{lnk1}, we have $\kappa\le C$ except on intervals $[a,b]$ of length less than $\delta$. On such an interval,
\begin{align*}
\kappa(\phi)-\frac{\lambda}{\kappa(\phi)}&=\kappa(a)-\frac{\lambda}{\kappa(a)}+\int^{\phi}_{a}\big(\kappa(\theta)
-\frac{\lambda}{\kappa(\theta)}\big)_{\theta}d\theta\\
&\le\kappa(a)-\frac{\lambda}{\kappa(a)}+\int^{\phi}_{a}\big\lvert\big(\kappa(\theta)
-\frac{\lambda}{\kappa(\theta)}\big)_{\theta}\big\lvert d\theta\\
&\le\kappa(a)-\frac{\lambda}{\kappa(a)}+\sqrt{\delta}\bigg[\int^{2\pi}_{0}\big(\kappa(\theta)
-\frac{\lambda}{\kappa(\theta)}\big)^{2}_{\theta}d\theta\bigg]^{\frac{1}{2}}\\
&\le C-\frac{\lambda}{C}+\sqrt{\delta}\bigg[\int^{2\pi}_{0}\big(\kappa-\frac{\lambda}{\kappa}\big)^{2}d\theta+D\bigg]^{\frac{1}{2}},
\end{align*}
where we have used the inequality \eqref{kkk}. Assume that $\kappa$ attains its maximum at $\phi$. Then one has
\begin{equation}\label{123}
\kappa_{\max}-\frac{\lambda}{\kappa_{\max}}\le C-\frac{\lambda}{C}+\sqrt{\delta}\big[2\pi\big(\kappa_{\max}-\frac{\lambda}{\kappa_{\max}}\big)^{2}+D\big]^{\frac{1}{2}}.
\end{equation}\label{223}
If $\kappa_{\max}\le C,$we have done; if $\kappa_{\max}> C,$ it is easy to obtain from \eqref{123} that
\begin{equation}\label{223}
(1-2\pi\delta)\kappa^{4}_{\max}+2\big(\frac{\lambda}{C}-C\big)\kappa^{3}_{\max}+\bigg[\big(\frac{\lambda}{C}-C\big)^{2}-\delta
(D-4\pi\lambda)-2\lambda\bigg]\kappa^{2}_{\max}-2\lambda\big(\frac{\lambda}{C}-C\big)\kappa_{\max}\le(2\pi\delta-1)\lambda^{2}.
\end{equation}
Choose $\delta$ small enough such that $1-2\pi\delta>0$, then there exists a positive constant $C^{\ast}$ such that $\kappa_{\max}<C^{\ast}$, or the left-hand side of \eqref{223} will tend to infinity when $\kappa_{\max}$ is large enough. Therefore $\kappa$ is uniformly bounded on $S^{1}\times[0,T^{\ast})$.
\end{proof}
Hence, we have the following theorem.
\begin{theorem}
If $\kappa:S^{1}\times[0,T^{\ast})\to\mathbb{R}$ satifies \eqref{kt}, then $\kappa$ is uniformly bounded on $S^{1}\times[0,T^{\ast})$.
\end{theorem}
We will now use the estimates for the $L^{2}$ norms of the derivatives of $\kappa$ with respect to $\theta$ to show that under the assumption that $\kappa$ is bounded on $[0,T^{\ast})$, all higher derivatives of $\kappa$ with respect to $\theta$ are also bounded. To simplify the notation, we represent the partial derivatives with respect to $\theta$ by $"'"$, for example,
$$\kappa'=\frac{\partial\kappa}{\partial\theta},\quad\kappa''=\frac{\partial^{2}\kappa}{\partial\theta^{2}},\cdots.$$
\begin{lemma}
If $\kappa$ is bounded on $[0,T^{\ast}),$ then $\kappa'$ is also bounded on the same interval.
\end{lemma}
\begin{proof}
By \eqref{kt}, we have
\begin{equation*}
\frac{\partial^{2}\kappa}{\partial t\partial\theta}=\big(\frac{\partial\kappa}{\partial t}\big)'=\big(\kappa^{2}+\lambda\big)\kappa'''+2\kappa\kappa'\kappa''-\frac{4\lambda}{\kappa}\kappa'\kappa''
+\frac{2\lambda}{\kappa^{2}}(\kappa')^{3}+3\kappa^{2}\kappa'-\lambda\kappa'.
\end{equation*}
Let $W=e^{\mu t}\kappa'$. Then one has
\begin{equation}\label{34}
 W_{t}=\mu e^{\mu t}\kappa'+e^{\mu t}(\kappa')_{t}=\big(\kappa^{2}+\lambda\big)W''+\big(2\kappa-\frac{4\lambda}{\kappa}\big)W'We^{-\mu t}
+ \big(\frac{2\lambda(\kappa')^{2}}{\kappa^{2}}+3\kappa^{2}-\lambda+\mu\big)W.
\end{equation}
Let $W_{\min}(t)=\inf\{W(\theta,t)|\theta\in[0,2\pi]\}$ and suppose by contradiction that there exists an $\xi$, satisfying $\xi<W_{\min}(0)=\kappa'_{\min}(0)\le0$, such that $W_{\min}(t)=\xi$ fo some $t>0$. Let $t^{\ast}=\inf\{t|W_{\min}(t)=\xi\}$. Then the continuity of $W$ assures that this minimum $\xi$ is achieved for the first time at $(\theta^{\ast},t^{\ast})$, where we have
$$W_{t}\le0,\quad W''\ge0,\quad W'=0,\quad W=\xi<0,$$
which contradicts that $W$ satisfies \eqref{34} if we choose $\mu$ to be a large negative number. Therefore we obtain for $t>0$
$$W_{\min}(t)\ge W_{\min}(0),$$
which implies $\kappa'(\theta,t)\ge\kappa'_{\min}(0)e^{-\mu T^{\ast}}$. Similarly, Let $W_{\max}(t)=\sup\{W(\theta,t)|\theta\in[0,2\pi]\}$, we can obtain that $W_{\max}(t)\le W_{\max}(0)$ for all $t>0$, i.e., $\kappa'(\theta,t)\le\kappa'_{\max}(0)e^{-\mu t}\le\kappa'_{\max}(0)$.

To sum up , $\kappa'$ remains bounded on $[0,T^{\ast})$.
\end{proof}

\begin{lemma}\label{k23}
If $\kappa$ and $\kappa'$ are are both bounded on $[0,T^{\ast})$, then $\int^{2\pi}_{0}(\kappa'')^{4}d\theta$ is bounded.
\end{lemma}
\begin{proof}
We compute
\begin{align*}
\frac{d}{dt}\int^{2\pi}_{0}(\kappa'')^{4}d\theta=&4\int^{2\pi}_{0}(\kappa'')^{3}\big[(\kappa^{2}+\lambda)\kappa''
-\frac{2\lambda}{\kappa}(\kappa')^{2}+\kappa^{3}-\lambda\kappa\big]'' d\theta\\
=&-12\int^{2\pi}_{0}(\kappa'')^{2}\kappa'''\big[(\kappa^{2}+\lambda)\kappa'''+2\kappa\kappa'\kappa''+\frac{2\lambda}{\kappa^{2}}
(\kappa')^{3}-\frac{4\lambda}{\kappa}\kappa'\kappa''+3\kappa^{2}\kappa'-\lambda\kappa'\big]d\theta\\
=&12\bigg[-\int^{2\pi}_{0}(\kappa^{2}+\lambda)(\kappa'')^{2}(\kappa''')^{2}d\theta-\int^{2\pi}_{0}\big(2\kappa\kappa'-
\frac{4\lambda}{\kappa}\kappa'\big)(\kappa'')^{3}\kappa'''d\theta\\
&-\int^{2\pi}_{0}\big[\frac{2\lambda}{\kappa^{2}}(\kappa')^{3}+3\kappa^{2}\kappa'-\lambda\kappa'\big](\kappa'')^{2}\kappa'''d\theta\bigg].
\end{align*}
By the bounds of $\kappa,\kappa'$ and $\lambda$ and the Peter-Paul inequality: for any $\epsilon>0$ there holds $ab\le\epsilon a^{2}+\frac{b^{2}}{4\epsilon},$ we have
\begin{align*}
-\int^{2\pi}_{0}\big(2\kappa\kappa'-\frac{4\lambda}{\kappa}\kappa'\big)(\kappa'')^{3}\kappa'''d\theta&\le c'_{1}\int^{2\pi}_{0}(\kappa'')^{4}d\theta+\int^{2\pi}_{0}\frac{c''_{1}}{4\epsilon_{1}}(\kappa'')^{2}(\kappa''')^{2}d\theta,\\
-\int^{2\pi}_{0}\big[\frac{2\lambda}{\kappa^{2}}(\kappa')^{3}+3\kappa^{2}\kappa'-\lambda\kappa'\big](\kappa'')^{2}\kappa'''d\theta
&\le c'_{2}\int^{2\pi}_{0}(\kappa'')^{2}d\theta+\int^{2\pi}_{0}\frac{c''_{2}}{4\epsilon_{2}}(\kappa'')^{2}(\kappa''')^{2}d\theta,
\end{align*}
where $\epsilon_{1}, \epsilon_{2}$ are positive constants, and
\begin{align*}
c'_{1}&=\epsilon_{1}c''_{1},\quad c''_{1}=\max\{(2\kappa\kappa'-\frac{4\lambda}{\kappa}\kappa')^{2}\},\quad c'_{2}=\epsilon_{2}c''_{2},\\
c''_{2}&=\max\big\{\big[\frac{2\lambda}{\kappa^{2}}(\kappa')^{3}+3\kappa^{2}\kappa'-\lambda\kappa'\big]^{2}\big\}.
\end{align*}
We also have
$$-\int^{2\pi}_{0}(\kappa^{2}+\lambda)(\kappa'')^{2}(\kappa''')^{2}d\theta\le-c'_{3}\int^{2\pi}_{0}
(\kappa'')^{2}(\kappa''')^{2}d\theta,$$
where $c'_{3}=\min\{\kappa^{2}+\lambda\}$. Choose $\epsilon_{1}, \epsilon_{2}$ appropriately such that $\frac{c''_{1}}{4\epsilon_{1}}+\frac{c''_{2}}{4\epsilon_{2}}=c'_{3}.$ Therefore,
\begin{align*}
\frac{d}{dt}\int^{2\pi}_{0}(\kappa'')^{4}d\theta&\le12\bigg[c'_{1}\int^{2\pi}_{0}(\kappa'')^{4}d\theta+c'_{2}\int^{2\pi}_{0}(\kappa'')^{2}d\theta\bigg]\\
&\le12c'_{1}\int^{2\pi}_{0}(\kappa'')^{4}d\theta+12c'_{2}\cdot\sqrt{2\pi}\bigg[\int^{2\pi}_{0}(\kappa'')^{4}d\theta\bigg]^{\frac{1}{2}}\\
&=c_{1}\int^{2\pi}_{0}(\kappa'')^{4}d\theta+c_{2}\bigg[\int^{2\pi}_{0}(\kappa'')^{4}d\theta\bigg]^{\frac{1}{2}},
\end{align*}
where $c_{1}=12c'_{1}, c_{2}=12\sqrt{2\pi}c'_{2}$. We see that $\int^{2\pi}_{0}(\kappa'')^{4}d\theta$ grows at most exponentially and therefore remains finite on a finite interval.
\end{proof}

\begin{lemma}
If $\kappa, \kappa'$ and $\int^{2\pi}_{0}(\kappa'')^{4}d\theta$ are bounded, then so is $\int^{2\pi}_{0}(\kappa''')^{2}d\theta$.
\end{lemma}
\begin{proof}
By a direct computation, we have
\begin{align*}
\frac{d}{dt}\int^{2\pi}_{0}(\kappa''')^{2}d\theta=&2\int^{2\pi}_{0}\kappa'''\big[(\kappa^{2}+\lambda)\kappa''
-\frac{2\lambda}{\kappa}(\kappa')^{2}+\kappa^{3}-\lambda\kappa\big]''' d\theta\\
=&2\bigg[-\int^{2\pi}_{0}(\kappa^{2}+\lambda)(\kappa^{(4)})^{2}d\theta-\int^{2\pi}_{0}(4\kappa\kappa'-\frac{4\lambda}{\kappa}\kappa')
\kappa'''\kappa^{(4)}d\theta\\
&-\int^{2\pi}_{0}(2\kappa-\frac{4\lambda}{\kappa})(\kappa'')^{2}\kappa^{(4)}d\theta
-\int^{2\pi}_{0}\big(6\kappa(\kappa')^{2}-\frac{4\lambda}{\kappa^{3}}(\kappa')^{4}\big)\kappa^{(4)}d\theta\\
&-\int^{2\pi}_{0}\big(2(\kappa')^{2}+\frac{10\lambda}{\kappa^{2}}(\kappa')^{2}+3\kappa^{2}-\lambda\big)\kappa''\kappa^{(4)}d\theta\bigg].
\end{align*}
By adopting the same trick as in the proof of Lemma \ref{k23} to bound the last four terms, we can prove that
$\int^{2\pi}_{0}(\kappa''')^{2}d\theta$ grows at most exponentially.
\end{proof}

\begin{corollary}
Under the same hypothesis as above, $\kappa''$ is bounded.
\end{corollary}
\begin{proof}
Recall that for any one-dimensional smooth function $f$, there holds
$$\max\lvert f\lvert^{2}\le C\int \lvert f'\lvert^{2}+f^{2}.$$
We apply this to $\kappa''$ to get the desired result.
\end{proof}

\begin{lemma}
If $\kappa, \kappa'$ and $\kappa''$ are uniformly bounded, then so are $\kappa'''$ and all the higher derivatives of $\kappa$.
\end{lemma}
\begin{proof}
Compute,
\begin{align*}
\frac{\partial\kappa'''}{\partial t}=&(\kappa^{2}+\lambda)\kappa^{(5)}+\big(6\kappa\kappa'-\frac{4\lambda}{\kappa}\kappa'\big)\kappa^{(4)}
+\big(8\kappa\kappa''+6(\kappa')^{2}+\frac{14\lambda}{\kappa^{2}}(\kappa')^{2}-\frac{12\lambda}{\kappa}\kappa''
+3\kappa^{2}-\lambda\big)\kappa'''\\
&-6\kappa'(\kappa'')^{2}+\frac{12\lambda}{\kappa^{4}}(\kappa')^{5}-\frac{36\lambda}{\kappa^{3}}(\kappa')^{3}\kappa''
+\frac{24\lambda}{\kappa^{2}}\kappa'(\kappa'')^{2}+6(\kappa')^{3}+18\kappa\kappa'\kappa''.
\end{align*}
Since $\kappa, \kappa', \kappa''$ and $\lambda$ are bounded, the maximum principle can be applied to
 $\kappa'''e^{\mu t}$ for suitably choosing $\mu$. This implies that $\kappa'''$ is bounded.

Generally, if $\kappa, \kappa', \cdots, \kappa^{(l-1)}$ are bounded, then
\begin{align*}
\frac{\partial\kappa^{(l)}}{\partial t}=&\big(\kappa^{2}+\lambda\big)\kappa^{(l+2)}+\big(2l\kappa\kappa'-\frac{4\lambda}{\kappa}\kappa'\big)\kappa^{(l+1)}
+\big(2\kappa\kappa''+l(l-1)\kappa\kappa''+l(l-1)(\kappa')^{2}-\frac{4l\lambda}{\kappa}\kappa''\\
&+\frac{2(2l+1)\lambda}{\kappa^{2}}(\kappa')^{2}+3\kappa^{2}-\lambda\big)\kappa^{(l)}+f(\kappa,\cdots,\kappa^{(l-1)}).
\end{align*}
This shows that $\kappa^{(l)}$ is bounded on any finite intervals.
\end{proof}
Therefore we have
\begin{theorem}
The  flow \eqref{fl2} exists in time interval $[0,\infty)$.
\end{theorem}

\section{$C^{\infty}$ Convergence}
Theorem \ref{hasc} tells us that the evolving curve converges to a circle in the Hausdorff sense($C^{0}$ Convergence). Mimicking the proofs of Secs.5.1-5.6 of \cite{Gage-Hamilton}, we can show that $\frac{\kappa_{\min}(t)}{\kappa_{\max}(t)}\to1$ as $t\to\infty$. This can be considered as the $C^{2}$ convergence.
In this section, we will use the method in \cite{xing} to prove that the evolving curve converges to a circle in the $C^{\infty}$ sense and finish the proof of Main Theorem.

\begin{lemma}\label{cru}
Under the flow, the quantity $\int^{2\pi}_{0}(\kappa')^{2}d\theta$ is uniformly bounded.
\end{lemma}
\begin{proof}
By lemma \ref{k2tj}, we have
\begin{equation*}
\int^{2\pi}_{0}\big(\kappa-\frac{\lambda}{\kappa}\big)^{2}_{\theta}d\theta\le
\int^{2\pi}_{0}\big(\kappa-\frac{\lambda}{\kappa}\big)^{2}d\theta+D.
\end{equation*}
Since $\kappa$ and $\lambda$ are both uniformly bounded, there exists a positive constant $\bar{D}$ such that
$$\int^{2\pi}_{0}(\kappa')^{2}d\theta<\int^{2\pi}_{0}\big(1+\frac{\lambda}{\kappa^{2}}\big)(\kappa')^{2}d\theta
=\int^{2\pi}_{0}\big(\kappa-\frac{\lambda}{\kappa}\big)^{2}_{\theta}d\theta\le\bar{D}.$$
The proof is done.
\end{proof}
The estimate of $\int^{2\pi}_{0}(\kappa')^{2}d\theta$ is a crucial step in proving convergence of the flow \eqref{fl2}. By Lemma \ref{cru} and the Cauchy-Schwarz inequality, for any $\theta_{1},\theta_{2}\in[0,2\pi]$,
$$\vert\kappa(\theta_{1},t)-\kappa(\theta_{2},t)\vert=\bigg\vert\int^{\theta_{2}}_{\theta_{1}}\kappa'd\theta\bigg\vert\le\sqrt{
\vert\theta_{1}-\theta_{2}\vert}\sqrt{\int^{2\pi}_{0}(\kappa')^{2}d\theta}\le\sqrt{\bar{D}}\vert\theta_{1}-\theta_{2}\vert
^{\frac{1}{2}},$$
which implies $\kappa(\theta,t)$ is equicontinuous. Since $\kappa$ and $\int^{2\pi}_{0}(\kappa')^{2}d\theta$ are both uniformly bounded, Ascoli-Arzela Theorem tells us that there is a convergent subsequence, denoted by $\kappa(\theta,t_{i})$, as $t_{i}$ tends to infinity. So we can obtain the following $C^{\infty}$ convergence.
\begin{theorem}
Under the flow \eqref{fl2}, we have
$$\lim\limits_{t\to\infty}\big\Vert\kappa(\theta,t)-\sqrt{\frac{\pi}{A(0)}}\big\Vert_{C^{n}(S^{1})}=0,\quad,n=0,1,2\cdots.$$
\end{theorem}

\textbf{Acknowledgements}
The authors would like to thank Chengyang Yi for useful discussions relating to this paper. This work is partially supported by Science and Technology Commission
of Shanghai Municipality (No. 22DZ2229014). The research is supported by Shanghai Key Laboratory of PMMP.

\textbf{Data availability}
Data sharing  not applicable to this article as no datasets were generated or analysed during the current study.

\textbf{Conflict of interest}
The authors have no conflicts of interest to declare that are relevant to the content of this article.

% \vskip 20cm \begin{center}\,\end{center}

\end{document}